\documentclass[10pt]{amsart}

\makeatletter
\def\blfootnote{\gdef\@thefnmark{}\@footnotetext}
\makeatother

\usepackage{epsfig}
\usepackage{graphics}
\usepackage{dcpic, pictexwd}
\usepackage{float}
\usepackage[colorlinks=true,
                    linkcolor=blue,
                    urlcolor=blue,
                    citecolor=blue,
                    anchorcolor=blue]{hyperref}
\usepackage{mathtools}
\usepackage{amsmath,amssymb}

\theoremstyle{plain}

\newtheorem*{theorem*}{Theorem}

\newtheorem{theorem}{Theorem}[section]

\newtheorem{lemma}[theorem]{Lemma}

\theoremstyle{remark}

\theoremstyle{Acknowledgments}

\newtheorem*{main}{Main Theorem}

\theoremstyle{definition}

 \def\R{{\mathbb{R}}}
 \def\Z{{\mathbb{Z}}}
 
 \def\N{{\mathbb{N}}}

\def\mod{{\rm Mod}}

\def\Pm{{\rm PMod}}

\def\Ends{{\rm Ends}}

 \def\Sym{{\rm Sym}}

\begin{document}
\blfootnote{\textup{2000} \textit{Mathematics Subject Classification}:
57M07, 20F05, 20F38}
\blfootnote{\textit{Keywords}:
Big mapping class groups, infinite surfaces, involutions, generating sets}
\newenvironment{prooff}{\medskip \par \noindent {\it Proof}\ }{\hfill
$\square$ \medskip \par}
    \def\sqr#1#2{{\vcenter{\hrule height.#2pt
        \hbox{\vrule width.#2pt height#1pt \kern#1pt
            \vrule width.#2pt}\hrule height.#2pt}}}
    \def\square{\mathchoice\sqr67\sqr67\sqr{2.1}6\sqr{1.5}6}
\def\pf#1{\medskip \par \noindent {\it #1.}\ }
\def\endpf{\hfill $\square$ \medskip \par}
\def\demo#1{\medskip \par \noindent {\it #1.}\ }
\def\enddemo{\medskip \par}
\def\qed{~\hfill$\square$}

 \title[Involution Generators of the Big Mapping Class Group] {Involution Generators of the Big Mapping Class Group}

\author[T{\"{u}}l\.{i}n Altun{\"{o}}z,       Mehmetc\.{i}k Pamuk, and O\u{g}uz Y{\i}ld{\i}z ]{T{\"{u}}l\.{i}n Altun{\"{o}}z,    Mehmetc\.{i}k Pamuk, and O\u{g}uz Y{\i}ld{\i}z}
\address{Faculty of Engineering, Ba\c{s}kent University, Ankara, Turkey} 
\email{tulinaltunoz@baskent.edu.tr} 
\address{Department of Mathematics, Middle East Technical University,
 Ankara, Turkey}
 \email{mpamuk@metu.edu.tr}
 \address{Department of Mathematics, Middle East Technical University,
 Ankara, Turkey}
  \email{oguzyildiz16@gmail.com}


\begin{abstract}  
Let $S=S(n)$ denote the infinite surface with $n$ ends, $n\in \mathbb{N}$,  accumulated by genus. For $n\geq6$, we show that the mapping class group of $S$ is topologically generated by five involutions. When $n\geq3$, it is topologically  generated by six involutions.  
\end{abstract}
\maketitle
  \setcounter{secnumdepth}{2}
 \setcounter{section}{0}
 
\section{Introduction}
Let $S$ be a second countable, connected, orientable surface with compact (possibly empty) boundary.  We say that $S$ is of \textit{finite type} if its fundamental group is finitely generated and \textit{infinite type} otherwise.  The \textit{mapping class group} of $S$, denoted by $\mod(S)$, is defined as the group of isotopy classes of orientation preserving self-homeomorphisms of $S$.  

Let us first consider finite type surfaces. The mapping class group of a surface of finite type has been  studied in details for many years.  Many  sets of generators are known.   It is now a classical result that $\mod(S)$ is generated by finitely many Dehn twists about non-separating simple closed curves~\cite{de,H,l3}. The study of algebraic properties of mapping class group, finding small generating sets, generating 
sets with particular properties, has been an active one leading to interesting developments. Wajnryb~\cite{w} showed that $\mod(S)$ can be generated by two elements given as a product of Dehn twists. 
As the group is not abelian, this is the smallest possible. Later, Korkmaz~\cite{mk2} showed that one of these generators can be taken as a Dehn twist, he also proved that $\mod(S)$ can be generated 
by two torsion elements. The third author showed that $\mod(S)$ is generated by two torsion elements of small orders~\cite{y1}.

 Let $g$ denote the genus of $S$.  Generating $\mod(S)$ by involutions was first considered by McCarthy and Papadopoulus~\cite{mp}.  They showed that for $g \geq3$,  $\mod(S)$  can be generated by infinitely many conjugates of a single involution.  In terms of generating by finitely many involutions,
 Luo~\cite{luo} showed that any Dehn twist about a non-separating simple closed curve can be written as a product six involutions, which in turn implies that $\mod(S)$ can be generated 
 by $12g + 6$ involutions. Brendle and Farb~\cite{bf} obtained a generating set of six involutions for $g \geq 3$. Following their work, Kassabov~\cite
{ka} showed that $\mod(S)$ can be generated 
 by four involutions if $g \geq 7$. Korkmaz~\cite{mk1} showed that $\mod(S)$ is generated by three involutions if $g \geq8$ and four involutions if $g \geq 3$. Also, the third author improved 
 his result showing that it is generated by three involutions if $g \geq 6$~\cite{y2}.

Infinite-type surfaces and their mapping class groups, also called big mapping class groups, have generated great interest in the last several years.  These big mapping class groups can be seen as limit objects of the mapping class groups of finite type surfaces.  While work has been done by many authors to show that mapping class groups of finite-type surfaces are generated by torsion elements, not much has been done for the infinite-type case.  The goal of this note is to investigate involution generators for big mapping class groups.

It is now well-known that the homeomorphism type of a finite-type surface is determined by the triple $(g, p, b)$, where $g \geq 0$ is the genus, and $p\geq 0$ is the number of punctures and $b\geq 0$ is the number of boundary components of the surface.  To give a similar classification result for  infinite type surfaces we should first define the space of ends of a surface.

An \textit{end} of a surface $S$ is the equivalence class of a nested sequence of connected subsurfaces $U_1\supset U_2\supset \ldots$ of $S$ with compact boundary and with the property that for any compact subsurface $K \subset S$, $K \cap U_r = \emptyset$ for high enough $r$. Two such sequences $U_1 \supset U_2 \supset \ldots$ and $V_1 \supset V_2 \supset \ldots$ are equivalent if for every $r\in \N$ there exists $s \in \N$ such that $V_s \subset U_r$ and vice versa.
An end given by a sequence $U_1\supset U_2\supset \ldots$ is said to be \textit{accumulated by genus (nonplanar)} if every $U_r$ has positive genus. Otherwise, it is said to be a \textit{planar end}.

\textit{The space of ends} of $S$, denoted by $\Ends(S)$, is a topological space whose points are the ends of $S$ and whose basic open sets correspond to subsurfaces 
$U \subset S$ with compact boundary. An end $U_1\supset U_2\supset \ldots$ of $S$ is contained in the basic open set corresponding to $U$ if $U_n \subset S$ for high enough $n$.  By construction, $\Ends(S)$ is compact, separable, and totally disconnected—in other words, it is homeomorphic to a closed subspace of a Cantor space. The set of ends accumulated by genus is denoted by $\Ends_{\infty}(S)$ and is always a closed subspace of $\Ends(S)$.  By work of Richards~\cite{ric}, an orientable,
boundaryless, infinite-type surface, $S$, is completely classified by its (possibly infinite) genus,
its space of ends, $\Ends(S)$, and the closed subset of ends which are accumulated by genus, $\Ends_{\infty}(S)$.

Throughout the paper, we consider surfaces with infinite genus and $n$ ends, $n\in \mathbb{N}$, and all ends are accumulated by genus. Let us denote such a surface by $S(n)$.  The \textit{pure mapping class group}, denoted by  $\Pm(S(n))$, is the subgroup of $\mod(S(n))$ fixing $\Ends(S)$ pointwise.  For the involution generators of big mapping class groups of some other infinite type surfaces we refer the reader to \cite[Theorem 1.2, Theorem 1.3]{ah} where the author obtained generating sets with the minimum possible number of generators.     

In the case of surfaces of infinite type, $\mod(S(n))$ is not countably generated.  On the other hand, being a quotient of the group of orientation self-homeomorphism of $S(n)$ (that is equipped with the compact open topology),  $\mod(S(n))$ inherits a topology. This makes $\mod(S(n))$ a Polish group \cite[Proposition 4.1]{av}, in particular $\mod(S(n))$ is separable.  Hence, $\mod(S(n))$ is topologically generated by a countable set i.e., there is a countable set that generates a dense subgroup (see \cite{av} for more details).  


The pure mapping class group $\Pm(S(n))$ is a normal subgroup of $\mod(S(n))$ of index $n!$. For $n\geq 2$, we have the following exact sequence:
 \[
1\longrightarrow \Pm(S(n))\longrightarrow \mod(S(n)) \longrightarrow \Sym_{n}\longrightarrow 1,
\]
where $\Sym_n$ is the symmetric group on $n$ letters and the last projection is given by the restriction of the isotopy class of a diffeomorphism to its action on ends.

Results on mapping class groups of finite-type surfaces do not immediately  extend for infinite-type surfaces. However, when it comes to generating the mapping class group, Patel and Vlamis \cite[Theorem 4]{pv} showed that the pure mapping class group of a surface is topologically generated by Dehn twists if the surface has at most one end accumulated by genus, and by Dehn twists and maps called \textit{handle shifts} otherwise. 

Now, let us immediately define handle shifts.  Consider the surface $S$ obtained by taking $\R\times [0,1]$, removing the interior of each  disk of radius $\frac{1}{4}$ for each $n\in \Z$, and gluing a torus with one boundary component to the boundary of each disk. Define a homeomorphism $\sigma$ that acts like $(x,y)\mapsto (x+1,y)$ on the interior of $S$ and extends as the identity homeomorphism in a neighbourhood of $\partial S$. Roughly speaking, the homeomorphism $\sigma$ slides the $n^{\textit{th}}$- handle (in the disk centered at $(n,\frac{1}{2})$) horizontally until it comes to the position of $(n+1)^{\textit{th}}$- handle (in the disk centered at $(n+1,\frac{1}{2})$) (see Figure~\ref{H}).

\begin{figure}[hbt!]
\begin{center}
\scalebox{0.28}{\includegraphics{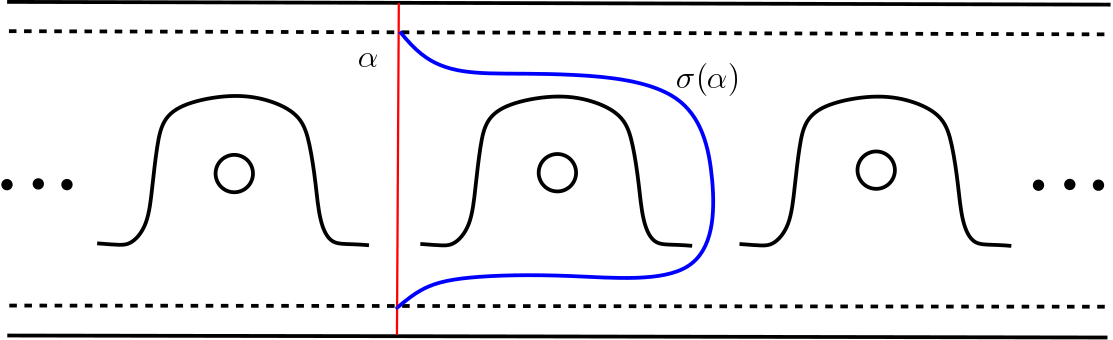}}
\caption{The homeomorphism $\sigma$ on the surface$S$.}
\label{H}
\end{center}
\end{figure}

Before we state the main theorem of the paper, we want to note here that by the results of Malestein and Tao\cite[Theorem A]{malestein-tao} for a uniformly self-similar surface (surface that has a self-similar ends space with infinitely many maximal ends and zero or infinite genus) $\mod(S)$ is generated by involutions, normally generated by a single involution. Also in \cite{domat}, the authors prove that the closure of the compactly supported mapping class group of an infinite-type surface is not generated by the collection of multitwists.

In this paper, we obtain the following main result:
\begin{main}\label{main}
For $n\geq 6$, $\mod(S(n))$ is topologically generated by five involutions and for $n\geq 3$, it is topologically generated by six involutions.
\end{main}

 Before we finish this section, let us fix our notation. Throughout the paper we do not distinguish a 
 diffeomorphism from its isotopy class. For the composition of two diffeomorphisms, we
use the functional notation; if $f$ and $g$ are two diffeomorphisms, then
the composition $fg$ means that $g$ is applied first and then $f$.\\
\indent
 For a simple closed curve $a$ on $\mod(S(n))$, we denote the right-handed Dehn twist $t_a$ about $a$ by the corresponding capital letter $A$.  We  denote inverse of any mapping class $X$ by $\overline{X}$.  

Finally, let us recall the following basic facts of Dehn twists that we use frequently in the rest of the paper. Let $a$ and $b$ be two
simple closed curves on $S(n)$ and $f\in \mod(S(n))$.
\begin{itemize}
\item  If $a$ and $b$ are disjoint, then $AB=BA$ (\textit{Commutativity}).
\item If $f(a)=b$, then $fA \overline{f}=B$.  (\textit{Conjugation}).
\end{itemize}

\noindent
 
\medskip

\noindent
{\bf Acknowledgements.}
This work is supported by the Scientific and Technological Research Council of Turkey (T\"{U}B\.{I}TAK)[grant number 120F118].


\section{Proof of the Main Theorem}\label{S3}
Let us start with reminding  the following basic fact from group theory.
\begin{lemma}\label{lemma1}
Let $G$ and $K$ be groups. Suppose that the following short exact sequence holds,
\[
1 \longrightarrow N \overset{i}{\longrightarrow}G \overset{\pi}{\longrightarrow} K\longrightarrow 1.
\]
Then a subgroup $\Gamma$ of $G$ satisfies $i(N)\subseteq \Gamma$ and $\pi(\Gamma)=K$ if and only if $\Gamma=G$.
\end{lemma}
\par

In our case where $G=\mod(S(n))$ and $N=\Pm(S(n))$, the following short exact sequence holds for $n\geq 2$:
\[
1\longrightarrow \Pm(S(n)) \overset{i} \longrightarrow \mod(S(n)) \overset{\pi}\longrightarrow \Sym_{n}\longrightarrow 1.
\]
Hence, we get the following useful result which follows immediately from Lemma~\ref{lemma1}: If $\Gamma$ is a subgroup of $\mod(S(n))$ with $\Pm(S(n))\subseteq \Gamma$ and $\pi(\Gamma)=\Sym_n$, then $\Gamma=\mod(S(n))$.


\begin{figure}[hbt!]
\begin{center}
\scalebox{0.4}{\includegraphics{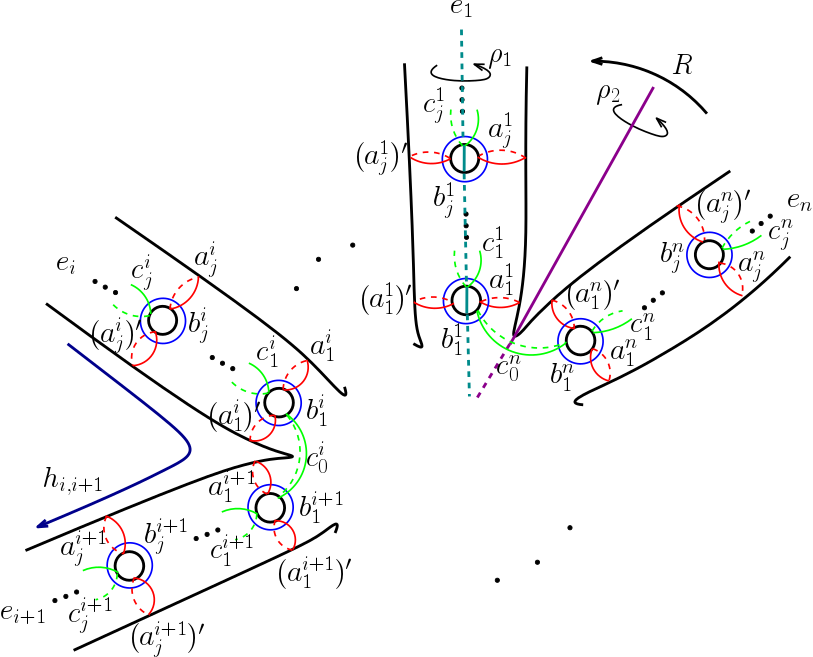}}
\caption{The curves $a_j^{i},b_j^{i},c_j^{i},$ and $c_0^{i} $, the rotations $\rho_1$, $\rho_2$, $R$ and the handle shift $h_{i,i+1}$ for $i=1,2,\ldots,n$ on $S(n)$.}
\label{G}
\end{center}
\end{figure}

Now, consider the model for $S(n)$ depicted in Figure~\ref{G}. If $n\geq 2$, there is a handle shift $h_{i,i+1}$, whose action can be described as
\[
 \begin{array}{lcr}
  \vspace*{0.2cm}
 h_{i,i+1}(b_1^{i})=b_1^{i+1},   & \hspace*{2cm} 
 h_{i,i+1}(a_1^{i})=(a_1^{i+1})',   & \hspace*{2cm} h_{i,i+1}(c_0^{i})=c_1^{i+1},  \\
 \vspace*{0.2cm}
 h_{i,i+1}(b_{j\neq 1}^{i})=b_{j-1}^{i},   & \hspace*{2cm} 
 h_{i,i+1}(a_{j\neq 1}^{i})=a_{j-1}^{i},  & \hspace*{2cm} h_{i,i+1}(c_{j\neq 0}^{i})=c_{j-1}^{i},\\
h_{i,i+1}(b_j^{i+1})=b_{j+1}^{i+1},   & \hspace*{2cm} 
 h_{i,i+1}(a_j^{i+1})=a_{j+1}^{i+1},   & \hspace*{2cm} h_{i,i+1}(c_j^{i+1})=c_{j+1}^{i+1}.
 \end{array} 
\]
Note that the surface $S(n)$ is invariant under the rotations $\rho_1$ and $\rho_2$ which are the $\pi$ rotations about the indicated lines shown in Figure~\ref{G}. Moreover, the homeomorphism $R=\rho_1\rho_2$ is the rotation by $\frac{2\pi}{n}$, which satisfies 
\[
R(\alpha^i)=\alpha^{i+1}, \textrm{ where } \alpha \in \lbrace a_k,b_k,c_{k-1} \rbrace \textrm{ for } k={1,2,\ldots }
\]

\begin{figure}[hbt!]
\begin{center}
\scalebox{0.35}{\includegraphics{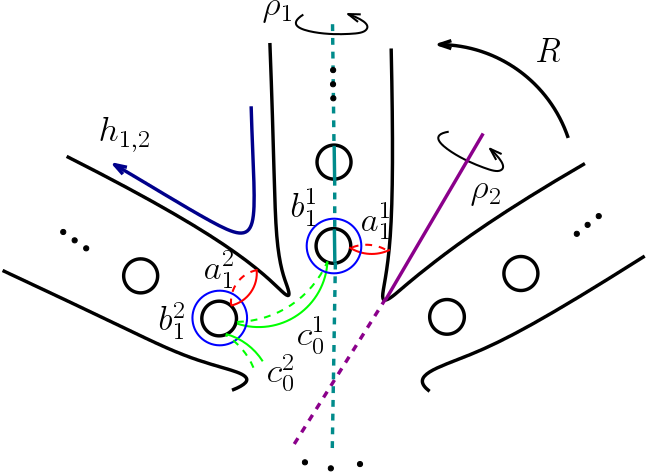}}
\caption{The curves $a_1^{1}, a_1^{2},b_1^{1}, b_1^{2},c_0^{1}$ and $c_0^{2} $ on $S(n)$.}
\label{lantern}
\end{center}
\end{figure}

\begin{lemma}\label{lem33}
For $n\geq3$, the group topologically generated by the elements
\[
\lbrace \rho_1, \rho_2, A_1^{1}\overline{A_1^{2}},B_1^{1}\overline{B_1^{2}}, C_0^{1}\overline{C_0^{2}}, h_{1,2}\rbrace
\]
contains the Dehn twists $A_1^{2}\overline{A_2^{2}}$, $B_1^{2}\overline{B_2^{2}}$ and $C_1^{2}\overline{C_2^{2}}$.
\end{lemma}
\begin{proof}
Let $G$ denote the subgroup topologically generated by the elements 
\[
\lbrace \rho_1, \rho_2, A_1^{1}\overline{A_1^{2}},B_1^{1}\overline{B_1^{2}}, C_0^{1}\overline{C_0^{2}}, h_{1,2}\rbrace.
\]
Since the rotation $\rho_1$ sends the curves $(a_1^{1},a_1^{2})$ to the curves  $((a_1^{1})',(a_1^{n})')$, we have
\[
(A_1^{1})'\overline{(A_1^{n})'}=(A_1^{1}\overline{A_1^{2}})^{\rho_1}\in G.
\]
Then it follows from $h_{1,2}((a_1^{1})',(a_1^{n})')=(a_1^{2},(a_1^{n})')$ that we have
\[
A_1^{2}\overline{(A_1^{n})'}=\big((A_1^{1})'\overline{(A_1^{n})'}\big)^{h_{1,2}}\in G.
\]
Moreover, since $h_{1,2}(a_1^{2},(a_1^{n})')=(a_2^{2},(a_1^{n})')$, we obtain
\[
A_2^{2}\overline{(A_1^{n})'}=\big(A_1^{2}\overline{(A_1^{n})'}\big)^{h_{1,2}}\in G.
\]
Hence, we have
\[
A_1^{2}\overline{A_2^{2}}=\big(A_1^{2}\overline{(A_1^{n})'}\big)\big((A_1^{n})'\overline{ A_2^{2}}\big)\in G.
\]
It follows from $h_{1,2}(b_1^{1},b_1^{2})=(b_1^{2},b_2^{2})$ that we get
\[
B_1^{2}\overline{B_2^{2}}=(B_1^{1}\overline{B_1^{2}})^{h_{1,2}}\in G.
\]
Note that the rotation $R=\rho_1\rho_2\in G$. Since $R(c_0^{1},c_0^{2})=(c_0^{2},c_0^{3})$, the following element
\[
C_0^{2}\overline{C_0^{3}}=(C_0^{1}\overline{C_0^{2}})^{R}\in G
\]
From this, we have
\[
C_0^{1}\overline{C_0^{3}}=(C_0^{1}\overline{C_0^{2}})(C_0^{2}\overline{C_0^{3}})\in G.
\]
Moreover, since
\begin{eqnarray*}
h_{1,2}(c_0^{1},c_0^{3})&=&(c_1^{2},c_0^{3})
 \textrm{ if } n\neq 3\\
\big(h_{1,2}(c_0^{1},c_0^{3})&=&(c_1^{2},c_0^{2})
 \textrm{ if } n=3\big),
  \end{eqnarray*}
we obtain that the element
\begin{eqnarray*}
C_1^{2}\overline{C_0^{3}}&=&(C_0^{1}\overline{C_0^{3}})^{h_{1,2}}\in G \textrm{ if } n\neq 3\\ 
\big( C_1^{2}\overline{C_0^{3}}&=&(C_0^{1}\overline{C_0^{3}})^{h_{1,2}}(C_0^{2}\overline{C_0^{3}})\in G \textrm{ if } n=3 \big)
\end{eqnarray*}
Then the element
\[
C_0^{1}\overline{C_1^{2}}=(C_0^{1}\overline{C_0^{3}})(C_0^{3}\overline{C_1^{2}})\in G.
\]
Finally, since $h_{1,2}(c_0^{1},c_1^{2})=(c_1^{2},c_2^{2})$, we have
\[
C_1^{2}\overline{C_2^{2}}=(C_0^{1}\overline{C_1^{2}})^{h_{1,2}}\in G.
\]
This completes the proof.
\end{proof}

\begin{figure}[hbt!]
\begin{center}
\scalebox{0.35}{\includegraphics{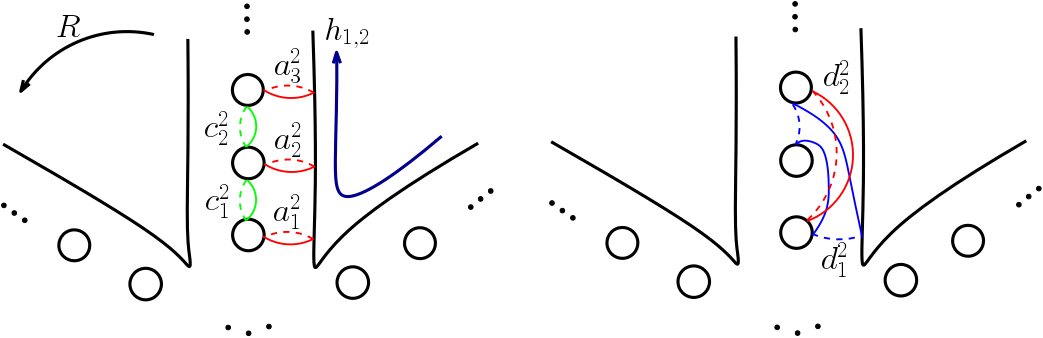}}
\caption{The curves of the embedded lantern relation $A_1^{2}C_1^{2}C_2^{2}A_3^{2}=A_2^{2}D_1^{2}D_2^{2}$ on $S(n)$.}
\label{lantern}
\end{center}
\end{figure}
\begin{lemma}\label{lem44}
For $n\geq3$, the group topologically generated by the elements
\[
\lbrace \rho_1, \rho_2, A_1^{2}\overline{A_2^{2}},B_1^{2}\overline{B_2^{2}}, C_1^{2}\overline{C_2^{2}}, h_{1,2}\rbrace
\]
contains the Dehn twists $A_i^{j}$, $B_i^{j}$, $C_{i-1}^{j}$ for all $j=1,2,\ldots, n$ and for all $i=1, 2, 3, \ldots$

\end{lemma}
\begin{proof}
Let $G$ be the subgroup generated by the elements
\[
\lbrace \rho_1, \rho_2, \rho_1, \rho_2, A_1^{2}\overline{A_2^{2}},B_1^{2}\overline{B_2^{2}}, C_1^{2}\overline{C_2^{2}}, h_{1,2} \rbrace.
\]
Let us now denote the curves $a_i^{2}$, $b_i^{2}$ and $c_i^{2}$ by $a_i$, $b_i$ and $c_i$, and also corresponding Dehn twists by $A_i$, $B_i$ and $C_i$, respectively, so that the subgroup $G$ contains the elements $A_1\overline{A_2}$, $B_1\overline{B_2}$ and $C_1\overline{C_2}$.

Since $h_{1,2}(a_1,a_2)=(a_2,a_3)$, the element $A_2\overline{A_3}$ is in the subgroup $G$. This also implies that
\[
A_1\overline{A_3}=(A_1\overline{A_2})(A_2\overline{A_3})\in G.
\]
It can be verified that
\[
(A_1\overline{A_2})(B_1\overline{B_2})(a_1,a_3)=(b_1,a_3)
\]
so we have $B_1\overline{A_3}\in G$ and 
since  
\[
(B_1\overline{B_2})(C_1\overline{C_2})(b_1,a_3)=(c_1,a_3)
\]
we also have $C_1\overline{A_3}\in G$.
Moreover, we have the following elements
\begin{eqnarray*}
B_2\overline{A_1}&=&(B_2\overline{B_1})(B_1\overline{A_3})(A_3\overline{A_2})(A_2\overline{A_1}),\\
C_1\overline{A_1}&=&(C_1\overline{A_3})(A_3\overline{A_2})(A_2\overline{A_1}),\\
C_2\overline{A_1}&=&(C_2\overline{C_1})(C_1\overline{A_1}) \textrm{ and }\\
A_2\overline{C_2}&=&(A_2\overline{A_1})(A_1\overline{C_2}),
\end{eqnarray*}
which are all contained in the subgroup $G$.
It can be checked that
\[
(B_2\overline{A_1})(C_1\overline{A_1})(A_1\overline{A_2})(C_2\overline{A_1})(b_2,a_1)=(d_1,a_1),
\]
where the curve $d_1$ is as in Figure~\ref{lantern}. Since each factor on the left hand side and $B_2\overline{A_1}$ are contained in $G$, we get $D_1\overline{A_1}\in G$.
Since $h_{1,2}(b_1,b_2)=(b_2,b_3)$, the element $B_2\overline{B_3}\in G$. Hence,  we have the element
\[
B_3\overline{A_1}=(B_3\overline{B_2})(B_2\overline{A_1})\in G.
\]
It can also be verified that
\[
(B_3\overline{A_1})(C_2\overline{A_1})(A_3\overline{A_1})(B_3\overline{A_1})(d_1,a_1)=(d_2,a_1),
\]
where the curve $d_1$ is as in Figure~\ref{lantern}. Again, since $G$ contains each factors and $D_1\overline{A_1}$, it also contains  $D_2\overline{A_1}$. This implies that
\[
D_2\overline{C_1}=(D_2\overline{A_1})(A_1\overline{C_1})\in G.
\]
Now, using the following lantern relation (see Figure~\ref{lantern}), 
\[
A_1C_1C_2A_3=A_2D_1D_2
\]
we obtain
\[
A_3=(A_2\overline{C_2})(D_1\overline{A_1})(D_2\overline{C_1})\in G,
\]
since each factor is contained in $G$. Thus using the actions of the elements $R=\rho_1\rho_2$ and $h_{1,2}$, we get all Dehn twists $A_i^{j}\in G$ for all $j=1,2,\ldots,n$.
Moreover, the subgroup $G$ contains
\begin{eqnarray*}
C_1&=&(C_1\overline{A_1})A_1\in G \textrm{ and }\\
B_1&=&(B_1\overline{A_3})A_3\in G.
\end{eqnarray*}
By conjugating these elements with $R=\rho_1\rho_2$ and $h_{1,2}$, we conclude that $B_i^{j}$ and $C_i^{j}$ are all contained in the subgroup $G$ for all $j=1,2,\ldots,n$. Also it follows from  $\overline{h_{1,2}}(c_1^{2})=C_0^{1}$ that we get 
\[
C_0^{1}=(C_1^{2})^{\overline{h_{1,2}}}\in G,
\]
which implies $C_0^{j}$ is contained in $G$ for each $j=1,2,\ldots,n$ by the action of $R$. This finishes the proof.
\end{proof}
Using lemmata~\ref{lem33} and~\ref{lem44}, we immediately conclude the following theorem.
\begin{theorem}\label{lemthm}
For $n\geq3$, the group topologically generated by the elements
\[
\lbrace \rho_1, \rho_2, A_1^{1}\overline{A_1^{2}},B_1^{1}\overline{B_1^{2}}, C_0^{1}\overline{C_0^{2}}, h_{1,2}\rbrace
\]
contains the Dehn twists $A_i^{j}$, $B_i^{j}$, $C_{i-1}^{j}$ for all $j=1,2,\ldots, n$ and for all $i=1, 2, 3, \ldots$ 
\end{theorem}

\begin{figure}[hbt!]
\begin{center}
\scalebox{0.45}{\includegraphics{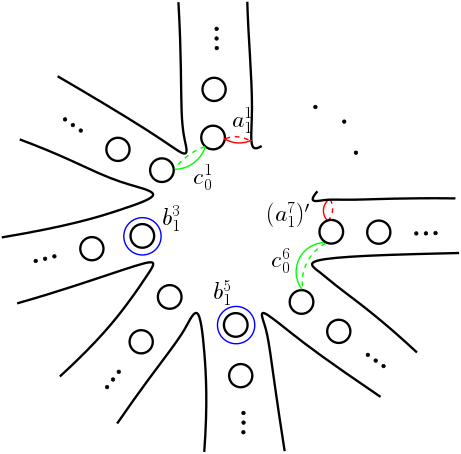}}
\caption{The curves $a_1^{1}, c_0^{1}, b_1^{3}, b_1^{5}, c_0^{6},a_1^{7'}$ on $S(n)$ for $n\geq 7$.}
\label{G7}
\end{center}
\end{figure}
\begin{lemma}\label{lem4} For $n\geq 7$, the group topologically generated by the elements
\[
\lbrace  \rho_1, \rho_2, A_1^{1}C_0^{1}B_1^{3}\overline{B_1^{5}}\overline{C_0^{6}}\overline{A_1^{7'}}, h_{1,2}\rbrace
\]
contains the Dehn twists $A_i^{j}$, $B_i^{j}$, $C_i^{j}$ and $C_0^{j}$ for all $j=1,2,\ldots,n$.
\end{lemma}
\begin{proof}
Let $F_1$ denote the element $A_1^{1}C_0^{1}B_1^{3}\overline{B_1^{5}}\overline{C_0^{6}}\overline{A_1^{7'}}$ and $H$ be the subgroup topologically generated by the elements $\lbrace \rho_1, \rho_2, F_1, h_{1,2}\rbrace$. Note that the rotation $R=\rho_1\rho_2$ belongs to the subgroup $H$. By Theorem~\ref{lemthm}, it suffices to prove that the elements $A_1^{1}\overline{A_1^{2}}$, $B_1^{1}\overline{B_1^{2}}$ and $C_0^{1}\overline{C_0^{2}}$ are contained in $H$.

For simplicity, let us denote the Dehn twists $A_1^{i}$,  $A_1^{i'}$,  $B_1^{i}$,  $C_0^{i}$ by $A_i$,  $A_{i}'$,  $B_{i}$,  $C_{i}$, respectively.

\noindent
Let $F_2$ be the element obtained by conjugation of $F_1$ with $R$. Hence
\begin{eqnarray*}
F_2&=&F_1^{R}=A_{2}C_{2}B_{4}\overline{B_{6}}\overline{C_{7}}\overline{A_{8}'} \in H\textrm{ if } n>7\\
(F_2&=&F_1^{R}=A_{2}C_{2}B_{4}\overline{B_{6}}\overline{C_{7}}\overline{A_{1}'} \in H\textrm{ if } n=7).\\
\end{eqnarray*}
It follows from 
\begin{eqnarray*}
F_2F_1(a_2,c_2,b_4,b_6,c_7,a_8')&=&(a_2,b_3,b_4,c_6,c_7,a_8') \textrm{ if } n>7\\
(F_2F_1(a_2,c_2,b_4,b_6,c_7,a_1')&=&(a_2,b_3,b_4,c_6,c_7,a_1') \textrm{ if } n=7)\\
\end{eqnarray*}
that
\begin{eqnarray*}
F_3&=&F_2^{F_2F_1}=A_2B_3B_4\overline{C_6}\overline{C_7}\overline{A_8'}\in H\textrm{ if } n>7\\
(F_3&=&F_2^{F_2F_1}=A_2B_3B_4\overline{C_6}\overline{C_7}\overline{A_1'}\in H\textrm{ if } n=7).
\end{eqnarray*}
The subgroup $H$ contains the element
\[
F_4=F_3^{\overline{R}}=A_1B_2B_3\overline{C_5}\overline{C_6}\overline{A_7'}.
\]
Since $F_4F_3(a_1,b_2,b_3,c_5,c_6,a_7')=(a_1,a_2,b_3,c_5,c_6,a_7')$, we get the following element:
\[
F_5=F_4^{F_4F_3}=A_1A_2B_3\overline{C_5}\overline{C_6}\overline{A_7'}\in H.
\]
Thus, the element $F_5\overline{F_4}=A_2\overline{B_2}\in H$, which implies that $A_1\overline{B_1}=(A_2\overline{B_2})^{\overline{R}}\in H$ and $A_3\overline{B_3}=(A_2\overline{B_2})^{R}\in H$.
The elements
\begin{eqnarray*}
F_6&=&(A_3\overline{B_3})F_1=(A_3\overline{B_3})(A_1C_1B_3\overline{B_5}\overline{C_6}\overline{A_7'})=(A_1C_1A_3\overline{B_5}\overline{C_6}\overline{A_7'})\in H \textit{ and }\\
F_7&=&F_6^{R}=A_2C_2A_4\overline{B_6}\overline{C_7}\overline{A_8'} \in H\textit{ if } n>7\\
(F_7&=&F_6^{R}=A_2C_2A_4\overline{B_6}\overline{C_7}\overline{A_1'} \in H \textit{ if } n=7).
\end{eqnarray*}
Since 
\begin{eqnarray*}
F_7F_6(a_2,c_2,a_4,b_6,c_7,a_8')=(a_2,c_2,a_4,c_6,c_7,a_8')\textit{ if } n>7\\
(F_7F_6(a_2,c_2,a_4,b_6,c_7,a_1')=(a_2,c_2,a_4,c_6,c_7,a_1')\textit{ if } n=7),\\
\end{eqnarray*}
we obtain the element
\begin{eqnarray*}
F_8&=&F_7^{F_7F_6}=A_2C_2A_4\overline{C_6}\overline{C_7}\overline{A_8'}\in H\textit{ if } n>7\\
(F_8&=&F_7^{F_7F_6}=A_2C_2A_4\overline{C_6}\overline{C_7}\overline{A_1'}\in H\textit{ if } n=7).\\
\end{eqnarray*}
Then we have $\overline{F_7}F_8=B_6\overline{C_6}\in H$, which leads to $B_i\overline{C_i}\in H$ for $i=1,2,\ldots,n$ by the action of $R$.
Moreover, it follows from $\rho_1(b_1,c_1)=(b_1,c_n)$ and 
$R(b_1,c_n)=(b_2,c_1)$ that we get $B_1\overline{C_n}\in H$ and $B_2\overline{C_1}\in H$.
We conclude that the subgroup $H$ contains the following elements:
\begin{eqnarray*}
C_1\overline{C_2}&=&(C_1\overline{B_2})(B_2\overline{C_2}),\\
B_1\overline{B_2}&=&(B_1\overline{C_1})(C_1\overline{B_2}) \textrm{ and}\\
A_1\overline{A_2}&=&(A_1\overline{B_1})(B_1\overline{B_2})(B_2\overline{A_2}),
\end{eqnarray*}
which completes the proof.
\end{proof}

\begin{figure}[hbt!]
\begin{center}
\scalebox{0.35}{\includegraphics{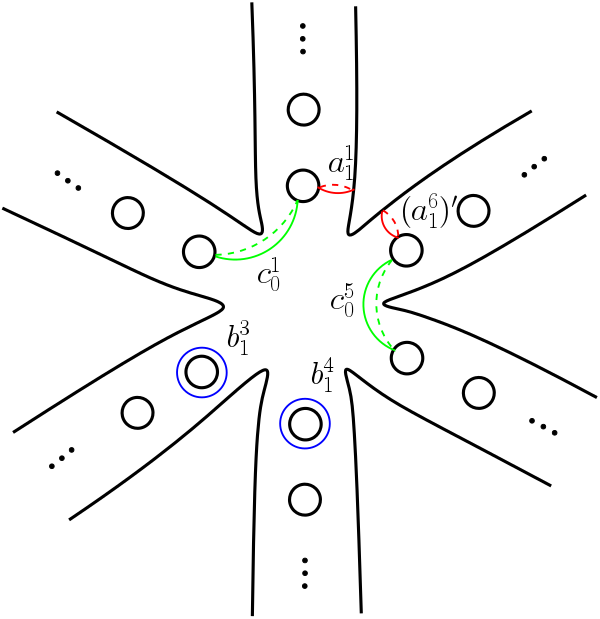}}
\caption{The curves $a_1^{1}, c_0^{1}, b_1^{3}, b_1^{4}, c_0^{5},a_1^{6'}$ on $S(6)$.}
\label{G6}
\end{center}
\end{figure}
\begin{lemma}\label{lem5} For $n=6$, the group topologically generated by the elements
\[
\lbrace  \rho_1, \rho_2, A_1^{1}C_0^{1}B_1^{3}\overline{B_1^{4}}\overline{C_0^{5}}\overline{A_1^{6'}}, h_{1,2}\rbrace
\]
contains the Dehn twists $A_i^{j}$, $B_i^{j}$, $C_i^{j}$ and $C_0^{j}$ for all $j=1,2,\ldots,n$.
\end{lemma}
\begin{proof}
Let $K_1$ denote the element $ A_1^{1}C_0^{1}B_1^{3}\overline{B_1^{4}}\overline{C_0^{5}}\overline{A_1^{6'}}$. Let $K$ denote the subgroup topologically generated by the elements $\lbrace \rho_1, \rho_2, K_1, h_{1,2}\rbrace$. Note that the rotation $R=\rho_1\rho_2\in K$. By Theorem~\ref{lemthm}, it is enough to show that the subgroup $K$ contains the elements $A_1^{1}\overline{A_1^{2}}$, $B_1^{1}\overline{B_1^{2}}$ and $C_0^{1}\overline{C_0^{2}}$.

Again for simplicity, let us denote the Dehn twists $A_1^{i}$,  $A_1^{i'}$,  $B_1^{i}$,  $C_0^{i}$ by $A_i$,  $A_{i}'$,  $B_{i}$,  $C_{i}$, respectively.

\noindent
The subgroup $K$ contains the element
\[
K_2=K_1^{R}=A_2C_2B_4\overline{B_5}\overline{C_6}\overline{A_1'}.
\]
It can be verified that
\begin{eqnarray*}
K_2K_1(a_2,c_2,b_4,b_5,c_6,a_1')&=&
(a_2,b_3,b_4,c_5,c_6,a_1') \textrm{ and  }\\
K_1K_2(a_1,c_1,b_3,b_4,c_5,a_6')&=&(a_1,c_1,c_2,b_4,b_5,a_6').
\end{eqnarray*}
We then have the following elements:
\begin{eqnarray*}
K_3&=&K_2^{K_2K_1}=A_2B_3B_4\overline{C_5}\overline{C_6}\overline{A_1'}\in K\textrm{ and  }\\
K_4&=&K_1^{K_1K_2}=A_1C_1C_2\overline{B_4}\overline{B_5}\overline{A_6'}\in K.
\end{eqnarray*}
The following elements are also contained in $K$:
\begin{eqnarray*}
K_5&=&K_3^{R^2}=A_4B_5B_6\overline{C_1}\overline{C_2}\overline{A_3'}\\
K_6&=&\overline{K_5}=A_3'C_2C_1\overline{B_6}\overline{B_5}\overline{A_4}.
\end{eqnarray*}
It is easy to see that
\[
K_6K_4(a_3',c_2,c_1,b_6,b_5,a_4)=(a_3',c_2,c_1,a_6',b_5,b_4),
\]
which implies that
\[
K_7=K_6^{K_6K_4}=A_3'C_2C_1\overline{A_6'}\overline{B_5}\overline{B_4}\in K.
\]
From this, we get
\begin{eqnarray*}
K_8&=&K_4\overline{K_7}=A_1\overline{A_3'}\in K \textrm{ and }\\
K_9&=&\overline{K_8}=A_3'\overline{A_1}\in K 
\end{eqnarray*}
Since $K_9K_3(a_3',a_1)=(b_3,a_1)$, we obtain the element
\[
K_{10}=B_3\overline{A_1}\in K.
\]
Moreover, we have the following elements:
\begin{eqnarray*}
K_{11}&=&K_{10}^{R^2}=B_5\overline{A_3}\in K \textrm{ and }\\
K_{12}&=&\overline{K_{11}}=A_3\overline{B_5}\in K.
\end{eqnarray*}
It can be checked that 
\[
K_{12}K_{10}(a_3,b_5)=(b_3,b_5), \]
which implies that 
\[
K_{13}=B_3\overline{B_5}\in K.
\]
Similarly, it follows from
\[
K_{13}K_{3}(b_3,b_5)=(b_3,c_5)
\]
that we get
\[
K_{14}=B_3\overline{C_5}\in K.
\]
The subgroup $K$ also contains the element
\[
K_{15}=K_{14}^{\rho_1}=B_5\overline{C_2}
\]
since $\rho_1(b_3,c_5)=(b_5,c_2)$.
The following elements are also contained in $K$:
\begin{eqnarray*}
B_1\overline{C_4}&=&K_{15}^{R^2}=(B_5\overline{C_2})^{R^2} \textrm{ and }\\
C_4\overline{B_2}&=&\overline{K_{14}}^{\overline{R}}=(C_5\overline{B_3})^{\overline{R}}.
\end{eqnarray*}
This implies that
\[
B_1\overline{B_2}=(B_1\overline{C_4})(C_4\overline{B_2})\in K.
\]
Also, the subgroup $K$ contains the element
\[
A_1\overline{A_2}=\overline{K_{10}}(B_1\overline{B_2})^{R^2}K_{10}^{R}=(A_1\overline{B_3})(B_3\overline{B_4})(B_4\overline{A_2}).
\]
It remains to prove that $C_1\overline{C_2}\in K$. Consider the following elements:
\begin{eqnarray*}
B_2\overline{B_3}&=&(B_1\overline{B_2})^{R},\\
B_1\overline{B_3}&=&(B_1\overline{B_2})(B_2\overline{B_3}) \textrm{ and }\\
C_4\overline{C_5}&=&\overline{(B_1\overline{C_4})}(B_1\overline{B_3})K_{14}=(C_4\overline{B_1})(B_1\overline{B_3})(B_3\overline{C_5}),
\end{eqnarray*}
which are all contained in $K$. This finishes the proof since we get
\[
C_1\overline{C_2}=(C_4\overline{C_5})^{R^3}\in K.
\]
\end{proof}
\begin{figure}[hbt!]
\begin{center}
\scalebox{0.35}{\includegraphics{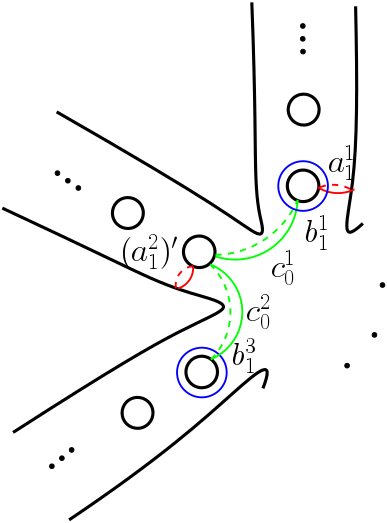}}
\caption{The curves $a_1^{1}, b_{1}^{1}, c_0^{1}, a_1^{2'}, c_0^{2}, b_1^{3}$ on $S(n)$ for $n\geq 3$}
\label{G3}
\end{center}
\end{figure}
\begin{lemma}\label{lem6} For $n\geq3$, the group topologically generated by the elements
\[
\lbrace  \rho_1, \rho_2, B_1^{1}C_0^{1}\overline{C_0^{2}}\overline{B_1^{3}}, A_1^{1}\overline{A_1^{2'}}, h_{1,2}\rbrace
\]
contains the Dehn twists $A_i^{j}$, $B_i^{j}$, $C_i^{j}$ and $C_0^{j}$ for all $j=1,2,\ldots,n$.
\end{lemma}
\begin{proof}
Let $L_1$ and $L_2$ denote the elements $B_1^{1}C_0^{1}\overline{C_0^{2}}\overline{B_1^{3}}$ and $A_1^{1}\overline{A_1^{2'}}$, respectively. Let $L$ denote the subgroup topologically generated by the elements $\lbrace \rho_1, \rho_2, L_1, L_2, h_{1,2}\rbrace$. It is clear that the rotation $R=\rho_1\rho_2\in L$. It follows from Theorem~\ref{lemthm} that we need to show that the subgroup $L$ contains the elements $A_1^{1}\overline{A_1^{2}}$, $B_1^{1}\overline{B_1^{2}}$ and $C_0^{1}\overline{C_0^{2}}$.

Let us again denote the Dehn twists $A_1^{i}$,  $A_1^{i'}$,  $B_1^{i}$,  $C_0^{i}$ by $A_i$,  $A_{i}'$,  $B_{i}$,  $C_{i}$, respectively.
The subgroup $L$ contains the elements
\begin{eqnarray*}
L_3&=&\overline{L_1}=B_3C_2\overline{C_1}\overline{B_1} \  \textrm{and}\\
L_4&=&L_2^{R}=A_2\overline{A_3'}.
\end{eqnarray*}
It can be shown that $L_2L_1(a_1,a_2')=(b_1,a_2')$, which implies that
\[
L_5=L_2^{L_2L_1}=B_1\overline{A_2'} \in L.
\]
Also we get the following elements
\begin{eqnarray*}
L_6&=&\overline{L_4}=A_3'\overline{A_2}\in L \textrm{ and}\\
L_7&=&L_5^{R}=B_2\overline{A_3'} \in L.
\end{eqnarray*}
It is easy to see that $L_1(b_1,a_2')=(c_1,a_2')$. This implies that
\[
L_8=L_5^{L_1}=C_1\overline{A_2'} \in L.
\]
Also, the subgroup $L$ contains
\[
L_9=L_8^{R}=C_2\overline{A_3'}.
\]
Since $L_6L_3(a_3',a_2)=(b_3,a_2)$, we obtain
\[
L_{10}=L_6^{L_6L_3}=B_3\overline{A_2}\in L.
\]
Moreover, the subgroup $L$ contains the following elements:
\begin{eqnarray*}
L_{11}&=&L_{10}^{\overline{R}}=B_2\overline{A_1} \ \textrm{and}\\
L_{12}&=&L_{11}L_2=(B_2\overline{A_1})(A_1\overline{A_2'})=B_2\overline{A_2'}.
\end{eqnarray*}
From these we have
\[
B_1\overline{B_2}=L_5\overline{L_{12}}=(B_1\overline{A_2'})(A_2'\overline{B_2})\in L.
\]
We also obtain 
\[
L_{13}=\overline{L_5}(B_1\overline{B_2})L_7=(A_2'\overline{B_1})(B_1\overline{B_2})(B_2\overline{A_3'})=A_2'\overline{A_3'}\in L.
\]
Hence, we have
\[
C_1\overline{C_2}=L_8L_{13}\overline{L_9}=(C_1\overline{A_2'})(A_2'\overline{A_3'})(A_3'\overline{C_2})\in L.
\]
Finally, the element
\[
A_1\overline{A_2}=\overline{L_{11}}(B_1\overline{B_2})^{R}L_{11}^{R}=(A_1\overline{B_2})(B_2\overline{B_3})(B_3\overline{A_2})\in L.
\]
This completes the proof.
\end{proof}


\subsection*{Involution generators}
Now, we present our involution generators of $\mod(S(n))$ for $n\geq 3$. We first express the handle shift $h_{1,2}$ as a product of two involutions (here we use the involutions that were introduced in~\cite{ah}). 

Consider the models of $S(n)$ which are obtained by decomposing $S(n)$ so that $S(n)=A\cup D_1 \cup D_2 \cup \cdots D_{n-2}$, where $A$ is an infinite surface with two ends accumulated by genus and $n-2$ boundary components and each $D_i$ is an infinite genus surface with one end accumulated by genus and one boundary component (see Figures~\ref{tau1} and \ref{tau2}). Note that the surface $S(n)$ is invariant under the two rotations $\tau_1$ and $\tau_2$, where they are rotations by $\pi$ about indicated lines shown in Figures~~\ref{tau1} and \ref{tau2}, respectively. Observe that $h_{1,2}=\tau_1\tau_2$.
\begin{figure}[hbt!]
\begin{center}
\scalebox{0.35}{\includegraphics{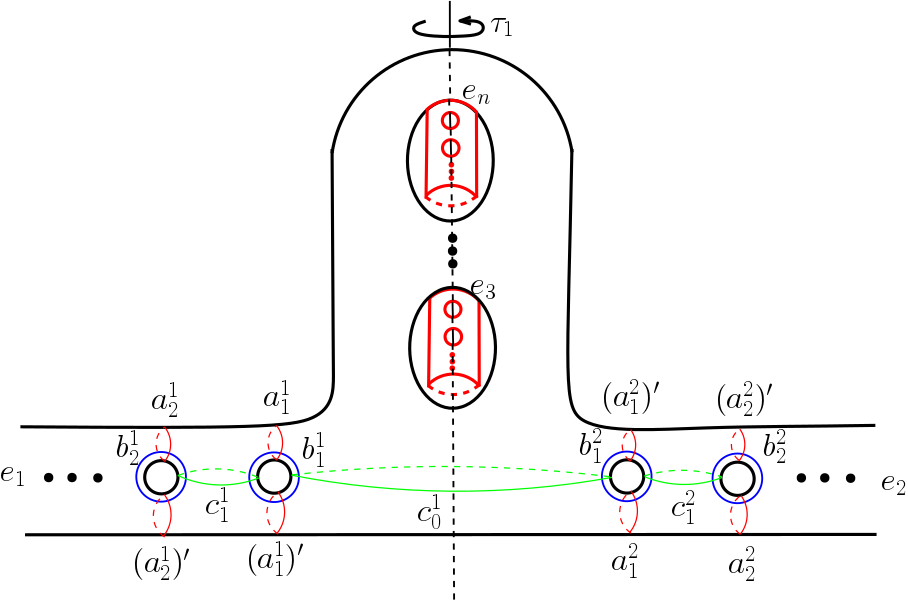}}
\caption{The involution $\tau_1$ on $S(n)$.}
\label{tau1}
\end{center}
\end{figure}
\begin{figure}[hbt!]
\begin{center}
\scalebox{0.28}{\includegraphics{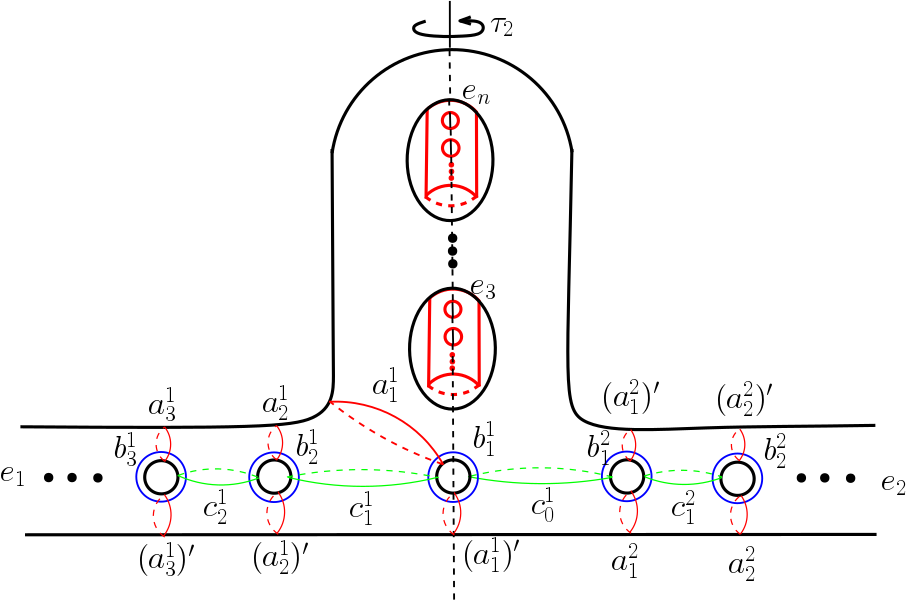}}
\caption{The involution $\tau_2$ on $S(n)$.}
\label{tau2}
\end{center}
\end{figure}

Consider now the surface $S(n)$ for $n\geq 7$ (see Figure~\ref{G7}). Since $\rho_3(a_1^{1})=(a_1^{7})'$, $\rho_3(c_0^{1})=c_0^{6}$ and $\rho_3(b_1^{3})=b_1^{5}$, where $R=\rho_1\rho_2$ and $\rho_3=R^3\rho_1R^{-3}$, we have 
\[
\rho_3A_1^{1}C_0^{1}B_1^{3}\overline{B_1^{5}}\overline{C_0^{6}}\overline{A_1^{7'}}
\]
is an involution. 

For the surface $S(6)$, it is easy to see that $\rho_2(a_1^{1})=(a_1^{6})'$, $\rho_2(c_0^{1})=c_0^{5}$ and $\rho_2(b_1^{3})=b_1^{4}$ (see Figure~\ref{G6}), which implies that 
\[
\rho_2A_1^{1}C_0^{1}B_1^{3}\overline{B_1^{4}}\overline{C_0^{5}}\overline{A_1^{6'}}
\]
is also an involution.

Let $\rho_4=R\rho_1R^{-1}$ and $\rho_5=R\rho_2R^{-1}$. For $n\geq3$, it can be observed that $\rho_4(b_1^{1})=b_1^{3}$, $\rho_4(c_0^{1})=c_0^{2}$, $\rho_5(a_1^{1})=(a_2^{2})'$ (see Figure~\ref{G3}). Hence, we get
\[
\rho_4B_1^{1}C_0^{1}\overline{C_0^{2}}\overline{B_1^{3}} \textrm{ and } \rho_5 A_1^{1}\overline{A_1^{2'}}
\]
are involutions.

The proof of the main theorem is now immediate: Let $\Gamma$ be the subgroup of $\mod(S(n))$ generated by the following involutions:
\begin{equation*}
 \begin{cases}
      \rho_1, \rho_2, \rho_3A_1^{1}C_0^{1}B_1^{3}\overline{B_1^{5}}\overline{C_0^{6}}\overline{A_1^{7'}}, \tau_1,\tau_2 & \text {if  $n\geq 7,$}\\
        \rho_1, \rho_2, \rho_2A_1^{1}C_0^{1}B_1^{3}\overline{B_1^{4}}\overline{C_0^{5}}\overline{A_1^{6'}}, \tau_1,\tau_2 &  \text{if $n= 6$},\\
                \rho_1, \rho_2, B_1^{1}C_0^{1}\overline{C_0^{2}}\overline{B_1^{3}}, A_1^{1}\overline{A_1^{2'}},  \tau_1,\tau_2 &  \text{if $n\geq 3$}.
    \end{cases}       
\end{equation*}
Note that $h_{1,2}=\tau_1\tau_2$ and $R=\rho_1\rho_1$ belong to $\Gamma$. It follows from lemmata~\ref{lem4}-\ref{lem6} that $\Gamma$ contains the Dehn twists $A_i^{j}$, $B_i^{j}$,$C_i^{j}$ and $C_0^{j}$ for all $j=1,2,\ldots,n$. Also, handle shifts are contained in $\Gamma$ by the conjugation of $R$. Hence, $\Pm(S(n))$ is contained in $\Gamma$ \cite[Theorem 4]{pv}. We finish the proof by showing that $\Gamma$ is mapped surjectively onto $\Sym_n$ by Lemma~\ref{lemma1}. The images of $R$ and $\tau_1$ are the $n$-cycle $(1,2,\ldots,n)$ and the $2$-cycle $(1,2)$, respectively. These elements generate $\Sym_n$, which completes the proof.



\begin{thebibliography}{xxxx}






\bibitem{av} J. Aramayona, N. Vlamis: \emph{Big mapping class groups: an overview}, In the tradition of Thurston: geometry and topology. Springer, (2020), 459--496.

\bibitem{bf} T. E. Brendle, B. Farb:  \emph{Every mapping class group is generated by $6$ involutions}, J. of Algebra {\bf{278}}, (1) (2004), 187--198.

\bibitem{de} M. Dehn:  \emph{The group of mapping classes},  In: Papers on Group Theory and Topology. Springer-Verlag, 1987. Translated from the German by J. Stillwell (Die Gruppe der Abbildungsklassen, Acta Math {\bf{69}}, (1938), 135--206).

\bibitem{domat} G. Domat, F. Fanoni, S. Hensel:  \emph{Multitwists in Big Mapping Class Groups}, arXiv:2301.08780v1 math. GT 20 Jan 2023 





\bibitem{H} S. Humphries:  \emph{ Generators for the mapping class group}, In: Topology of LowDimensional Manifolds, Proc. Second Sussex Conf., Chelwood Gate, (1977), Lecture Notes in Math. {\bf{722}}, (2) (1979), Springer-Verlag, 44--47.


\bibitem{ah} A. Huynh: \emph{Torsion in big mapping class groups}, Doctoral dissertation, Rutgers University-School of Graduate Studies (2022

\bibitem{ka} M. Kassabov:
\emph{Generating mapping class groups by involutions}, ArXiv
math.GT/0311455, v1 25Nov2003.






\bibitem{mk1} M. Korkmaz:
\emph{Mapping class group is generated by three involutions},
Math. Res. Lett.  {\bf{27}}, (4) (2020), 1095--1108.








\bibitem{mk2} M. Korkmaz:
\emph{Generating the surface mapping class group by two elements}, Trans. Amer. Math. Soc. {\bf{367}}, (8) (2005), 3299--3310.

\bibitem{l3} W. B. R. Lickorish:  \emph{A finite set of generators for the homeotopy group of a $2$-manifold}, Proc. Cambridge Philos. Soc. {\bf{60}}, (4) (1964), 769--778.

\bibitem{luo} F. Luo:  \emph{ Torsion elements in the mapping class group of a surface}, ArXiv math.GT/0004048, v1 8 Apr 2000.

\bibitem{malestein-tao} J. Malestein and J. Tao
\emph{Self-Similar Surfaces: Involutions and Perfection}, Michigan Math. J. Advance Publication, (2023), 1--24.

\bibitem{mp} J. D. McCarthy, A. Papadopoulos:  \emph{ Involutions in surface mapping class
groups}, Enseign. Math. {\bf{33}}, (2) (1987), 275--290.

\bibitem{pv} P. Patel and N. G. Vlamis: \emph{Algebraic and topological properties of big mapping class groups},  Algebr. Geom. Topol.{\bf{18}}, (2018): 4109--4142.

\bibitem{ric} I. Richards: \emph{On the classification of noncompact surfaces}, Trans. Amer. Math. Soc. {\bf{106}}, (1963), 259–-269.






\bibitem{w} B. Wajnryb:  \emph{Mapping class group of a surface is generated by two elements}, Topology
{\bf{35}}, (2) (1996), 377--383.




\bibitem{y1} O. Y{\i}ld{\i}z:
\emph{Generating mapping class group by two torsion elements}, Mediterr. J. Math. {\bf{19}}, (2) (2022), 59.



\bibitem{y2} O. Y{\i}ld{\i}z:
\emph{Generating mapping class group by three involutions}, ArXiv math.GT/2002.09151, 21Feb2020.















 






\end{thebibliography}
\end{document}